\documentclass[11pt,reqno]{amsart}
\usepackage{graphicx,amssymb,mathrsfs,amsmath,amsfonts,appendix}
\usepackage{subfigure}
\usepackage{amsthm}
\usepackage{paralist}
\usepackage{enumitem}
\usepackage[utf8]{inputenc}
\usepackage{comment}
\usepackage{amsmath}
\usepackage{mathtools}
\usepackage{esint}
\usepackage[colorlinks=true, pdfstartview=FitV, linkcolor=blue, citecolor=blue, urlcolor=blue]{hyperref}
\numberwithin{equation}{section}

\newtheorem{lemma}{Lemma}[section]
\newtheorem{proposition}{Proposition}[section]
\newtheorem{theorem}{Theorem}[section]

\usepackage[colorinlistoftodos]{todonotes}

\newcommand{\xd}{\textrm{d}}

\renewcommand{\div}[0]{\operatorname{div}}
\newcommand{\ep}[0]{\varepsilon}
\newcommand{\R}{\mathbb{R}}
 
\newcommand{\embed}{\hookrightarrow}
\def\wstarto{\stackrel{*}{\rightharpoonup}}

\allowdisplaybreaks

\title[Local asymptotics for
nonlocal Cahn-Hilliard equations]
{Local asymptotics\\
for nonlocal convective Cahn-Hilliard equations\\ with 
$W^{1,1}$ kernel and singular potential}
\author[Elisa Davoli]{Elisa Davoli}
\address{Institut f\"ur Mathematik, University of Vienna, Oskar-Morgenstern-Platz 1, 1090 Vienna, Austria}
\email{elisa.davoli@univie.ac.at}
\author[Luca Scarpa]{Luca Scarpa}
\address{Institut f\"ur Mathematik, University of Vienna, Oskar-Morgenstern-Platz 1, 1090 Vienna, Austria}
\email{luca.scarpa@univie.ac.at}
\author[Lara Trussardi]{Lara Trussardi}
\address{Institut f\"ur Mathematik, University of Vienna, Oskar-Morgenstern-Platz 1, 1090 Vienna, Austria}
\email{lara.trussardi@univie.ac.at}

\begin{document}

\begin{abstract}
We prove existence of solutions and study the nonlocal-to-local asymptotics for nonlocal, convective, Cahn-Hilliard equations in the case of a $W^{1,1}$ convolution kernel and under homogeneous Neumann conditions. 
Any type of potential, possibly also of double-obstacle or
logarithmic type, is included.
Additionally, we highlight variants and extensions to the setting of periodic boundary conditions and viscosity contributions, as well as connections with the general theory of evolutionary convergence of gradient flows.
\end{abstract}

\maketitle

\tableofcontents

\section{Introduction}
\label{sec:intro}

In this paper we continue the study of the nonlocal-to-local asymptotics of Cahn-Hilliard equations initiated in \cite{MRT18,DRST,DST}. In contrast to the results in \cite{MRT18,DRST,DST}, we focus here on the complementary case of long-range interaction kernels, 
possibly including the Riesz, Newtonian, and Bessel
potentials. In particular, we provide the first nonlocal-to-local convergence result for a class of convective Cahn-Hilliard equations, under homogeneous Neumann boundary conditions, and in the absence of regularizing viscosity terms.\\

 The Cahn-Hilliard equation is considered the pivotal model for spinodal decomposition: an irreversible process occurring in multiple composites (such as alloys, glasses, gels, ceramics, liquid solutions, and polymer solutions), and determined by local fluctuations in the concentrations of mixture components which eventually lead to a decomposition of the material into stable phases. This evolution equation was originally introduced in \cite{CH}, and has since then acquired a key role in materials science due to its aptitude to characterize a variety of different settings, ranging from biology to image reconstruction~\cite{CFM02,CM81,Tre03}.\\

The mathematical analysis of the classical Cahn-Hilliard
equation has been the subject of a very intense research activity in the past decades. Among the extensive literature, we mention the contributions \cite{cher-gat-mir, cher-pet} about existence and uniqueness of solution in domains with nonpermeable walls, the analysis in \cite{cher-mir-zel} including logarithmic double-well potentials, and the works in 
\cite{col-fuk-eqCH,gil-mir-sch} incorporating dynamic boundary conditions and irregular potentials, as well as all the references therein. A detailed analysis on the asymptotic behavior of solutions has been carried out in \cite{col-fuk-diffusion, col-scar} (see also the references therein).  

An augmented Cahn-Hilliard equation including additional convective contributions, and related to stirring of fluids and biological realizations of thin films is the subject of \cite{BDM18}. Relevant studies in coupling the Cahn-Hilliard equation with a further equation for the velocity field have been the focus of \cite{ab-dep-gar,ab-rog,boyer,gal-grass-CHNS}.\\

A key feature of the Cahn-Hilliard equation is the fact that it describes the $H^{-1}$-gradient flow of the Cahn-Hilliard-Modica-Mortola-energy functional, defined as
\begin{equation}
\label{eq:enCH}
\mathcal{E}_{CH}(\varphi):=\frac{1}{2}\int_{\Omega}|\nabla \varphi(x)|^2\,\xd x+\int_{\Omega}F(\varphi(x))\,\xd x,
\end{equation}
where $\varphi:\Omega\to \R$ denotes a concentration parameter, and where $F:\R\to \R$ is a suitable nonlinear double-well potential. With the above notation, and under the assumption of constant mobility (which, for simplicity, we assume identically equal to one), the Cahn-Hilliard equation reads as follows:
\begin{align}\label{eq:CH-general}
\begin{aligned}
    \partial_t \varphi-\Delta\mu
    =0 \qquad&\text{in } (0,T)\times\Omega\,,\\
    \mu\in-\Delta \varphi + \partial F(\varphi) \qquad&\text{in } (0,T)\times\Omega\,,\\
    \partial_{\bf n} \varphi=\partial_{\bf n} \mu=0
    \qquad&\text{in } (0,T)\times\partial\Omega\,,\\
    \varphi(0)=\varphi_{0} \qquad&\text{in } \Omega\,
    \end{aligned}
\end{align}
where $\mu$ represents the chemical potential associated to the energy $\mathcal{E}_{CH}$, $\partial F$ is the subdifferential (in the sense of convex analysis, see, e.g. \cite{barbu-monot}) of the double-well potential $F$, $\varphi_0$ is a suitable initial datum, $T>0$ is a fixed final time, and $\Omega$ is a smooth bounded domain in $\mathbb{R}^d$ (with $d=2,3)$. \\

In the early 90's a nonlocal counterpart of the above equation has been introduced by G. Giacomin and J. Lebowitz in \cite{GL} in order to provide a microscopical model of a $d$-dimensional lattice gas evolving via a Poisson nearest-neighbor process. The corresponding nonlocal Cahn-Hilliard equation is introduced as gradient flow of the nonlocal energy functional
\begin{equation}
    \label{NLenergy}
    \mathcal{E}^{NL}_\ep(\varphi) = \frac{1}{4}\int_\Omega\int_\Omega J_\ep(x,y)|\varphi(x)-\varphi(y)|^2\,\xd x\,\xd y+\int_\Omega F(\varphi(x))\,\xd x.
\end{equation}
 where, for every $\ep>0$, $J_\ep(x,y)$ is a positive and symmetric convolution kernel.

 The connection between the nonlocal and local gradient flows can be seen by a formal computation: in the case in which $J_{\ep}(x,y)=J_{\ep}(|x-y|)$ is suitably regular and concentrates around the origin as $\ep\to 0$, the nonlocal interface evolution approaches that in \eqref{eq:CH-general}. This observation is supported by the rigorous variational analysis of the asymptotic behavior of \eqref{NLenergy} as $\ep$ tends to zero. Indeed a whole nonlocal-to-local framework for functionals in the form \eqref{NLenergy} has been developed in the seminal papers by J. Bourgain, H. Brezis, and P. Mironescu \cite{BBM, BBM2}, as well as in the $\Gamma$-convergence analysis carried out by A.C. Ponce in \cite{ponce04, ponce}.\\

The interest in these nonlocal formulations is motivated by the fact that they exhibit a closer connection, compared to local models, to atomistic descriptions, and provide the ideal tool to describe pattern-formation phenomena.
The main novelty with respect to 
the local models is the presence of the possible long-range interaction kernel $J_\ep$, taking into account also the interaction between particles at a large scale (let's say $\ep^{-1}$).
As a result, nonlocal Cahn-Hilliard equations find applications in multiple settings, ranging from the modeling of tumor growth \cite{frig-lam-roc,roc-spr}, to the mechanisms describing phase transitions in polymer blends.

The recent years have witnessed an intense and 
increasing research activity on nonlocal 
Cahn-Hilliard equations:
we point out in this direction 
the contributions
\cite{ab-bos-grass-NLCH,
bat-han-NLCH,
gal-gior-grass-NLCH,
gal-grass-NLCH,han-NLCH}
and the references therein.
The main assumption on the interaction kernel 
in such studies is that $J_\ep$ is 
symmetric and of class $W^{1,1}$.
This is well motivated both in the direction of 
applications to diffuse interface modelling and 
from a mathematical perspective as well.
Indeed, on the one hand these assumptions allow
to consider all relevant examples of 
Newtonian and Bessel potentials, and on the other
hand they ensure enough integrability and regularity
on the 
solutions to the corresponding nonlocal evolutions
as it is expected in diffuse interface models.
\\

The variational convergence of the nonlocal energy 
to the local one gives rise naturally to 
the study of the asymptotics of the 
corresponding evolution problems. 
Such analysis has been initiated in some previous 
works of ours in different settings.
In \cite{MRT18}, assuming existence of the nonlocal evolution, the authors have shown nonlocal-to-local convergence of the Cahn-Hilliard equations in the case of polynomial double-well potentials satisfying a further concavity assumption. Very recently, in \cite{DRST, DST}, we have provided a first existence result for solutions to the nonlocal Cahn-Hilliard equations for singular kernels not falling within the $W^{1,1}$-existence theory and under possible degeneracy of the double-well potential. We worked under the assumption of constant mobility, first in the case of periodic boundary conditions, and afterwards with Neumann boundary conditions in the viscous (and vanishing viscosity) setting. Additionally, in \cite{DRST, DST} we have shown convergence to the classical Cahn-Hilliard equation, as the singular kernels $\{J_\ep\}_\ep$ concentrate around the origin (as $\ep\to 0$) and approach a Dirac delta. 
We point out that, aside from the two recent works \cite{SS1,SS2} dealing with the fractional Cahn-Hilliard equation, and from the general setting introduced in \cite{Gal18} (both not directly falling within our setting), the two papers \cite{DRST, DST} are the first contributions dealing with the case of non-regular interaction kernels.\\

The focus of this paper is on a counterpart to  the analysis in \cite{DRST, DST}. 
Indeed, in the mentioned papers we were forced
to choose interaction kernels with a singularity of 
order $2$ in the origin in order to guarantee
at least a suitable convergence of the energies.
Clearly, in dimension two and three such choice
does not ensure $W^{1,1}$ regularity.
As a consequence, the local asymptotics 
of nonlocal Cahn-Hilliard equations with 
$W^{1,1}$ kernels is currently an open problem:
we give here a first positive answer in this direction.
Indeed, we provide a full characterization of existence and nonlocal-to-local asymptotics in the case of singular double-well potential, interaction kernels satisfying $W^{1,1}$ integrability assumptions, and in the presence of a further convection term. 
The main idea of the work is to propose a different scaling of the interaction kernels, allowing 
to obtain both a suitable Gamma convergence of the energies and the required $W^{1,1}$ regularity.
From a modeling point of view, the additional regularity of the interaction kernel corresponds to enhancing the effects of long-range interactions
and the diffuse interface nature of the model. 
From a mathematical perspective, the $W^{1,1}$ integrability of $J_\ep$ allows to streamline all a-priori estimates involving integrations by parts. As a by-product, we are able here to provide the first nonlocal-to-local convergence result for a Cahn-Hilliard model involving homogeneous Neumann boundary conditions, and without any additional regularizing viscous terms
(as it was necessary in \cite{DST}).\\

In order to present our main results we need to introduce some basic notation.
We introduce below both the nonlocal system that we consider in this paper
\begin{align}
    \label{eq1_NL}
    \partial_t u_\ep-\Delta\mu_\ep
    =-\div (u_\ep{\bf v}) \qquad&\text{in } (0,T)\times\Omega\,,\\
    \label{eq2_NL}
    \mu_\ep\in(J_\ep*1)u_\ep - J_\ep*u_\ep + \gamma(u_\ep)
    +\Pi(u_\ep) \qquad&\text{in } (0,T)\times\Omega\,,\\
    \label{eq3_NL}
    \partial_{\bf n} \mu_\ep=0
    \qquad&\text{in } (0,T)\times\partial\Omega\,,\\
    \label{eq4_NL}
    u_\ep(0)=u_{0,\ep} \qquad&\text{in } \Omega\,,
\end{align}
and its local counterpart
\begin{align}
    \label{eq1_L}
    \partial_t u-\Delta\mu
    =-\div (u{\bf v}) \qquad&\text{in } (0,T)\times\Omega\,,\\
    \label{eq2_L}
    \mu\in-\Delta u + \gamma(u)
    +\Pi(u) \qquad&\text{in } (0,T)\times\Omega\,,\\
    \label{eq3_L}
    \partial_{\bf n} u=\partial_{\bf n} \mu=0
    \qquad&\text{in } (0,T)\times\partial\Omega\,,\\
     \label{eq4_L}
    u(0)=u_{0} \qquad&\text{in } \Omega.
\end{align}
Here $\gamma$ is a maximal monotone graph, and $\Pi$ is a Lipschitz map. Together, they form the subdifferential of a suitable double-well potential $F$. The map $\bf v$ represents a velocity field. The parameter $\ep$ identifies, roughly speaking, the amplitude of the range of nonlocal interactions. We refer to Section~\ref{sec:main} below for the precise regularity assumptions.\\

Our main results are the following. First, in Theorem \ref{thm1} we prove well-posedness of the system in \eqref{eq1_NL}--\eqref{eq4_NL}. Second, in Theorem \ref{thm2}, relying on the a-priori estimates identified in the proof of Theorem \ref{thm1}, we analyze nonlocal-to-local asymptotics. Our proof strategy is quite general: we collect in Section~\ref{sec:appl} both an application to the case of periodic boundary conditions, and an analysis of the viscous Cahn-Hilliard case, in which some regularity assumptions on the velocity field $\textbf{v}$ can be relaxed.
Due to the presence of the velocity field, our setting does not follow within the general theory of evolutionary $\Gamma$-convergence for gradient flows developed in \cite{sand-serf-2, sand-serf}. Additionally, even in the absence of convective contributions, our general assumptions on the double-well potential make the energy functional not $C^1$, thus situating our analysis outside the classical theory in \cite{sand-serf-2, sand-serf} but rather closer to the abstract metric framework in \cite{serf}. We devote the last section of this paper to show how some crucial estimates in the proof strategy of Theorems \ref{thm1} and \ref{thm2} directly relate to this general methodology.\\ 

The assumptions on the double-well potentials $F$ considered in our formulation are quite general. One the one hand, in fact, our class of double-well potentials includes the classical choice for $F$ as the fourth-order polynomial 
$F_{\rm pol}(r):=\frac14(r^2-1)^2$, $r\in\mathbb{R}$, with 
minima in $\pm1$ (corresponding to the pure phases). On the other hand, it also incorporated logarithmic 
double-well potentials, such as
$$F_{\log}(s)=\frac{\theta}{2}((1+s)\log(1+s)+(1-s)\log(1-s))-\frac{\theta_c}{2}s^2$$
for $0<\theta<\theta_c$, which by contrast is defined on
the bounded domain $(-1,1)$ and possesses minima within the open interval $(-1,1)$, and so-called double-obstacle potential (see \cite{BE, OoP}), having the form
$$
F_{\rm ob}(s)=I_{[-1,1]}(s)+\frac12 (1-s^2),\quad 
I_{[-1,1]}(s):=
\begin{cases}
0&\text{if }s\in [-1,1]\\+\infty&\text{otherwise}\,.
\end{cases}
$$
\\

The system in \eqref{eq1_NL}--\eqref{eq4_NL} is additionally driven by a convection term in divergence form, destroying the gradient-flow structure of the equation. Cahn-Hilliard diffusions under the action of convective terms play a central role both in the modeling of mixing and stirring of fluids, and in biological thin-films deposition via Langmuir-Blodgett transfer \cite{BL, Lang}. We recall here the recent works
\cite{BDM18,col-gil-spr-CHconv,EKZ13,WORD03} on
local Cahn-Hilliard models with convection,
\cite{della-grass,DPG16,roc-spr} studying nonlocal Cahn-Hilliard under local convection, and \cite{Ei87,IR07} on a nonlocal model with stellar convection. Phase separation in nonlocal convective Cahn-Hilliard systems is the subject of \cite{col-gil-spr-applTych, col-gil-spr-OCNLphase}. Further couplings of Cahn-Hilliard equations under evolving velocity fields have been analyzed in \cite{ab-dep-gar,ab-rog,boyer,gal-grass-CHNS}.\\

The paper is organized as follows: Section~\ref{sec:main} contains the mathematical setting and the precise statements of the main results. Section~\ref{sec:prelim} focuses on some useful preliminary results.
Section~\ref{sec:nonlocal} is devoted to the proof of well-posedness of the nonlocal system~\eqref{eq1_NL}--\eqref{eq4_NL}, while Section~\ref{sec:nonlocal-to-local} focuses on nonlocal-to-local asymptotics. Eventually, Section~\ref{sec:appl} and Section~ \ref{sec:SS} contain some generalizations of the results, and highlight their connections with evolutionary $\Gamma$-convergence of gradient flows, respectively. The main contributions of the paper are summarized in Section~\ref{sec:con}.

\section{Setting and main result}
\label{sec:main}

Throughout the paper, 
$\Omega\subset \mathbb{R}^d$ is a smooth bounded domain,
with $d=2,3$, and $T>0$ is a fixed final time.
We set $Q:=(0,T)\times\Omega$ and $Q_t:=(0,t)\times\Omega$ for all $t\in(0,T)$.
We define the functional spaces
\[
  H:=L^2(\Omega)\,, \quad
  V:=H^1(\Omega)\,, \quad V_0:=H^1_0(\Omega)\,,\quad
  W:=\{\varphi\in H^2(\Omega): \partial_{\bf n}\varphi=0
  \text{ a.e.~on } \partial\Omega\}\,,
\]
endowed with their natural norms $\|\cdot\|_H$,
$\|\cdot\|_V$, $\|\cdot\|_{V_0}$, 
and $\|\cdot\|_W$, respectively.
The duality pairing between $V^*$ and $V$
and the scalar product in $H$
will be denoted by $\langle\cdot,\cdot\rangle_V$
and $(\cdot, \cdot)_H$, respectively.
We also recall that the inverse of the $-\Delta$
 operator with 
homogeneous Neumann conditions
is a well-defined isomorphism
\[
  \mathcal N:\{\varphi\in V^*: \varphi_\Omega=0\}\to\{\varphi\in V:\varphi_\Omega=0\}\,, 
\]
where $\varphi_\Omega:=
\frac1{|\Omega|}
\langle\varphi,1\rangle$ 
for every $\varphi\in V^*$.

Throughout the work, we will consider the following assumptions.
\begin{description}
  \item[H1] $\alpha\in(0,d-1)$ is a fixed real number.
  \item[H2] $\rho:[0,+\infty)\to[0,+\infty)$
  is of class $C^1$, and such that 
  \[
  s\mapsto |\rho'(s)|
  s^{d-1-\alpha} \in L^1(\mathbb{R_+}),
  \]
  with the renormalization
  \[
  \int_0^{+\infty}\rho(s)s^{d+1-\alpha}\,\xd s=
  \frac2{C_d}\,, \qquad 
  C_d:=\int_{S^{d-1}}|\sigma\cdot e_1|^2\,\xd \mathcal{H}^{d-1}(\sigma)\,.
  \]
  The family of mollifiers $(\rho_\ep)_{\ep>0}$ is defined as
  \[
  \rho_\ep(r):=\frac1{\ep^d}\rho(r/\ep)\,, \quad r\geq0\,, \quad\ep>0\,.
  \]
  The convolution kernel $J_\ep:\mathbb{R}^d\to\mathbb R$ is given by
  \[
  J_\ep(z):=\frac1{\ep^{2-\alpha}}\rho_\ep(|z|)
  \frac1{|z|^\alpha}\,,
  \quad z\in\mathbb{R}^d\,,
  \]
  and we set 
  \[
  (J_\ep*\varphi)(x):=\int_\Omega J_\ep(x-y)\varphi(y)\,\xd y\,,
  \quad x\in\Omega\,,\quad\varphi\in L^1(\Omega)\,.
  \]
  \item[H3] $\gamma:\mathbb{R}\to 2^{\mathbb R}$ is a
  maximal monotone graph with $0\in\gamma(0)$, and 
  $\hat\gamma:\mathbb R\to [0,+\infty]$ is the unique 
  proper, convex and lower semicontinuous function
  such that $\partial\hat\gamma=\gamma$ and
  $\hat\gamma(0)=0$. For every $\lambda>0$, 
  the Yosida approximation of $\gamma$ is denoted 
  by $\gamma_\lambda:\mathbb R\to\mathbb R$.
  Moreover, $\Pi:\mathbb R\to \mathbb R$ is
  Lipschitz-continuous, and we set  
  $\hat\Pi:\mathbb R\to \mathbb R$ as
  $\hat\Pi(r):=\int_0^r\Pi(s)\,\xd s$, $r\in \mathbb R$.
  We will assume, with no loss of generality, that 
  $\hat\gamma + \hat\Pi \geq 0$, and that for every $\ep>0$
  there exist $c_\ep^0,\lambda_\ep^0>0$
  such that 
  \[
  \gamma_\lambda'(r)+\Pi'(r)+
  (J_\ep*1)(x) \geq c_\ep^0
  \quad\text{for a.e.~}(r,x)
  \in\mathbb R\times\Omega\,,\quad
  \forall\,\lambda\in(0,\lambda_\ep^0)\,.
  \]
  \item[H4] ${\bf v}\in L^2(0,T; (L^\infty(\Omega)\cap V_0)^d)$.
\end{description}
Assumption~\textbf{H1} is needed in order to guarantee that the kernel is of class $W^{1,1}$ in $\Omega$.
The assumptions on the mollifiers stated in~\textbf{H2} correspond to the requirements in~\cite{ponce04, ponce}; assumption~\textbf{H3} includes any type of double-well potential $F$ represented by $\hat{\gamma}+\hat{\Pi}$ such as the logarithmic, the the fourth-order polynomial but also the double-obstacle potential; finally we observe that our analysis includes quite general velocity fields $\bf v$, possibly varying both in space and time.

The local limiting energy functional is defined as
\[
  E:H\to[0,+\infty]\,,\qquad
  E(\varphi):=
  \begin{cases}
  \frac12\int_\Omega|\nabla\varphi(x)|^2\,\xd x
  \quad&\text{if } \varphi\in V\,,\\
  +\infty \quad&\text{if } \varphi\in H\setminus V\,.
  \end{cases}
\]
Furthermore, thanks to the assumption {\bf H2}, 
it follows that 
$J_\ep\in W^{1,1}(\mathbb{R}^d)$
for every $\ep>0$
(see Lemma~\ref{lem:rho} below). 
This implies in particular that 
the nonlocal energy satisfies
\begin{equation}
    \label{eq:nl-en}
  E_\ep(\varphi):=\frac14
  \int_{\Omega\times\Omega}
  J_\ep(x-y)|\varphi(x)-\varphi(y)|^2\,\xd x\,\xd y<+\infty
  \qquad\forall\,\varphi\in H\,.
\end{equation}
For further details we refer to 
Section~\ref{sec:prelim}.

The local problem is well-posed
(see \cite{col-gil-spr, col-gil-spr-CHconv}) in the 
following sense.
\begin{theorem}
  \label{thm_local}
  For every $u_{0}\in V$ such that $\hat\gamma(u_{0})\in L^1(\Omega)$
  and $(u_{0})_\Omega\in\operatorname{Int}
  D(\gamma)$, 
  there exists a triplet $(u, \mu, \xi)$,
  where $u$ is uniquely determined,
  such that 
  \begin{align*}
      &u \in H^1(0,T; V^*)
      \cap L^\infty(0,T; V)
      \cap L^2(0,T; W)\,,\\
      &\mu \in L^2(0,T; V)\,,\\
      &\xi\in L^2(0,T; H)\,,\\
      &u(0)=u_{0}\\
      &\mu=-\Delta u + \xi+
      \Pi(u)\,,\\
      &\xi\in\gamma(u) \quad\text{a.e.~in } Q\,,
  \end{align*}
  and
  \[
  \langle\partial_t u,\varphi\rangle_V + 
  \int_\Omega\nabla\mu(x)\cdot\nabla\varphi(x)\,\xd x = 
  \int_\Omega u(x){\bf v}(x)\cdot\nabla\varphi(x)\,\xd x
  \quad\forall\,\varphi\in V\,,
  \quad\text{a.e.~in } (0,T)\,.
  \]
\end{theorem}

The two main results of this paper are
the following: the first one deals with the existence and uniqueness of solutions to the nonlocal problem for $\ep$ fixed; the latter concerns the nonlocal-to-local convergence. 
\begin{theorem}
  \label{thm1}
  Let $\ep>0$ be fixed. Then, under assumptions \textbf{H1}--\textbf{H4},
  for every $u_{0,\ep}\in H$ such that $\hat\gamma(u_{0,\ep})\in L^1(\Omega)$
  and $(u_{0,\ep})_\Omega\in\operatorname{Int}
  D(\gamma)$, 
  there exists a triplet $(u_\ep, \mu_\ep, \xi_\ep)$, where $u_\ep$ is uniquely 
  determined,
  such that 
  \begin{align*}
      &u_\ep \in H^1(0,T; V^*)
      \cap L^2(0,T; V)\,,\\
      &\mu_\ep \in L^2(0,T; V)\,,\\
      &\xi_\ep\in L^2(0,T; V)\,,\\
      &u_\ep(0)=u_{0,\ep}\\
      &\mu_\ep=(J_\ep*1)u_\ep - J_\ep*u_\ep + \xi_\ep+
      \Pi(u_\ep)\,,\\
      &\xi_\ep\in\gamma(u_\ep) \quad\text{a.e.~in } Q\,,
  \end{align*}
  and
  \begin{equation}
  \label{eq:grad-flow}
  \langle\partial_t u_\ep,\varphi\rangle_V + 
  \int_\Omega\nabla\mu_\ep(x)\cdot\nabla\varphi(x)\,\xd x = 
  \int_\Omega u_\ep(x){\bf v}(x)\cdot\nabla\varphi(x)\,\xd x
  \quad\forall\,\varphi\in V\,,
  \quad\text{a.e.~in } (0,T)\,.
  \end{equation}
\end{theorem}

\begin{theorem}
  \label{thm2}
  Assume \textbf{H1}--\textbf{H4}, 
  let $(u_{0,\ep})_{\ep>0}\subset H$, and let
  $u_0\in V$
  be such that 
  \begin{align*}
      &\sup_{\ep>0}\left(
      E_\ep(u_{0,\ep}) + 
      \|\hat\gamma(u_{0,\ep})\|_{L^1(\Omega)}
      \right)<+\infty\,,\\
      &\exists\,[a_0,b_0]\subset
      \operatorname{Int}D(\gamma):\quad
      a_0\leq(u_{0,\ep})_\Omega\leq b_0 \quad
      \forall\,\ep>0\,,\\
      &u_{0,\ep}\to u_0
      \quad\text{in } H \quad\text{as }
      \ep\searrow0\,.
  \end{align*}
  Then, if $(u_\ep,\mu_\ep, \xi_\ep)$
  is the solution
  to the nonlocal problem
  given by Theorem~\ref{thm1},
  there exists a solution
  $(u,\mu,\xi)$ to the local problem 
  local problem such that 
  \begin{align}
      &\notag u_\ep\to u 
      \quad\text{in } C^0([0,T]; H)\,,\\
      &\label{eq1-f}u_\ep\wstarto u
      \quad\text{in } L^\infty(0,T; H)\cap H^1(0,T;V^*)\,,\\
      &\label{eq2-f}\mu_\ep\rightharpoonup\mu 
      \quad\text{in } L^2(0,T; V)\,,\\
      &\label{eq3-f}\xi_\ep\rightharpoonup\xi
      \quad\text{in } L^2(0,T; H)\,.
  \end{align}
\end{theorem}
\section{Preliminaries}
\label{sec:prelim}
We collect in this section some preliminary
results that will be used in the sequel. We first show that for our choice of $\alpha$ the kernel $J_\ep$ is of class $W^{1,1}$.

\begin{lemma}
  \label{lem:rho}
  For every $\ep>0$ it holds that $J_\ep\in W^{1,1}(\mathbb R^d)$.
\end{lemma}
\begin{proof}
  An elementary computation yields 
  \begin{align*}
      \int_{\mathbb R^d}J_\ep(z)\,\xd z&=
      \int_{\mathbb R^d}
      \frac1{\ep^{2-\alpha}}
      \frac1{\ep^d}
      \rho(|z|/\ep)\frac1{|z|^\alpha}\,\xd z
      =\frac1{\ep^{2}}
      \int_{\mathbb R^d}\rho(|w|)
      |w|^{-\alpha}\,\xd w\\
      &=
      \frac1{\ep^2}|S^{d-1}|\int_0^{+\infty}
      \rho(r)r^{d-1-\alpha}\,\xd r\,,
  \end{align*}
  where the right-hand side is finite 
  thanks to {\bf H1--H2}. 
  Hence, $J_\ep\in L^1(\mathbb R^d)$.
  Furthermore, note that for every 
  $i\in\{1,\ldots,d\}$,
  \[
  \partial_{z_i}J_\ep(z)=
  \frac1{\ep^{2-\alpha}}
  \frac{\rho_\ep'(|z|)\frac{z_i}{|z|}|z|^\alpha
  -\rho_\ep(|z|)\alpha|z|^{\alpha-1}\frac{z_i}{|z|}}
  {|z|^{2\alpha}}=
  \frac1{\ep^{d+2-\alpha}}
  \frac{\frac1{\ep}\rho'(|z|/\ep)z_i|z|
  -\alpha\rho(|z|/\ep)z_i}
  {|z|^{\alpha+2}}\,,
  \]
  yielding
  \begin{equation}
  \label{eq:grad}
  |\nabla J_\ep(z)|\leq\frac1{\ep^{d+2-\alpha}}
  \left(\frac1\ep\frac{|\rho'(|z|/\ep)|}{|z|^\alpha}
  +\alpha\frac{\rho(|z|/\ep)}{|z|^{\alpha+1}}\right)
  \,.
  \end{equation}
  The two terms on the right-hand side of \eqref{eq:grad}
  satisfy
  \[
  \frac1\ep\int_{\mathbb{R}^d}
  \frac{|\rho'(|z|/\ep)|}{|z|^\alpha}\,\xd z=
  \frac1{\ep^{\alpha+1-d}}
  \int_{\mathbb{R}^d}
  \frac{|\rho'(|w|)|}{|w|^\alpha}\,\xd w=
  \frac1{\ep^{\alpha+1-d}}|S^{d-1}|
  \int_0^{+\infty}|\rho'(r)|r^{d-1-\alpha}\,\xd r
  \]
  and 
  \[
  \alpha\int_{\mathbb{R}^d}
  \frac{\rho(|z|/\ep)}{|z|^{\alpha+1}}\,\xd z=
  \frac\alpha{\ep^{\alpha+1-d}}
  \int_{\mathbb{R}^d}
  \frac{\rho(|w|)}{|w|^{\alpha+1}}\,\xd w=
  \frac\alpha{\ep^{\alpha+1-d}}|S^{d-1}|
  \int_0^{+\infty}\rho(r)r^{d-2-\alpha}\,\xd r\,,
  \]
  so that by comparison we have
  \[
  \int_{\mathbb R^d}|\nabla J_\ep(z)|\,\xd z
  \leq\frac{|S^{d-1}|}{\ep^3}
  \int_0^{+\infty}\left(
  r|\rho'(r)| + \alpha\rho(r)
  \right)r^{d-2-\alpha}\,\xd r\,.
  \]
  The right-hand side is finite 
  again by {\bf H1--H2},
  hence $J_\ep\in W^{1,1}(\mathbb R^d)$.
\end{proof}

As we have anticipated, 
the fact that $J_\ep\in W^{1,1}(\mathbb{R}^d)$
implies that the nonlocal energy 
$E_\ep$ is well defined on the whole space $H$, i.e.~$E_\ep:H\to[0,+\infty)$.
Indeed, it easily follows
by the properties of the convolution
that 
\[
  E_\ep(\varphi)=
  \frac12\left((J_\ep*1)\varphi
  -J_\ep*\varphi,\varphi\right)_H\leq
  \|J_\ep\|_{L^1(\mathbb{R}^d)}
  \|\varphi\|_H^2
  \qquad\forall\,\varphi\in H\,.
\]

We proceed by studying the regularity of $E_\ep$, and by characterizing its differential.
\begin{lemma}
  \label{lem:diff}
  For all $\ep>0$, $E_\ep:H\to[0,+\infty)$
  is of class $C^1$ and $DE_\ep:H\to H$ is given by
  \begin{align*}
  \langle DE_\ep(\varphi),
  \zeta\rangle&=
  \int_\Omega \left((J_\ep*1)\varphi-
  J_\ep*\varphi\right)(x)\zeta(x)\,\xd x\\
  &=\frac12\int_{\Omega\times\Omega}
  J_\ep(x-y)(\varphi(x)-\varphi(y))(\zeta(x)-\zeta(y))
  \,\xd x\,\xd y
  \quad\varphi,\zeta\in H\,.
  \end{align*}
\end{lemma}
\begin{proof}
  The proof is a direct consequence of the 
  definition of G\^ateaux differentiability 
  and of the linearity of the convolution operator:
  see \cite{DRST, DST} for details.
\end{proof}
The next lemma provides a characterization of the asymptotic behavior of the energies $E_\ep$.
\begin{lemma}
  \label{lem:conv_E}
  For every $\varphi,\zeta\in V$ it holds that 
  \begin{align*}
  &\lim_{\ep\to0}E_\ep(\varphi)=E(\varphi)\,,\\
  &\lim_{\ep\to0}
  \int_\Omega
  \left((J_\ep*1)\varphi-
  J_\ep*\varphi\right)(x)\zeta(x)\,\xd x=
  \int_\Omega\nabla\varphi(x)
  \cdot\nabla\zeta(x)\,\xd x\,.
  \end{align*}
  Moreover, for every sequence
  $(\varphi_\ep)_{\ep>0}\subset H$ and
  $\varphi\in H$ it holds that
  \begin{align*}
      \sup_{\ep>0}E_\ep(\varphi_\ep)<+\infty
      \qquad&\Rightarrow\qquad
      (\varphi_\ep)_\ep \text{ is 
      relatively compact in } H\,,\\
      \varphi_\ep\to\varphi \quad\text{in } H
      \qquad&\Rightarrow\qquad
      E(\varphi)\leq
     \liminf_{\ep\to0}E_\ep(\varphi_\ep)\,.
  \end{align*}
\end{lemma}
\begin{proof}
  Note that
  \[
  E_\ep(\varphi)=
  \frac14\int_{\Omega\times\Omega}
  \rho_\ep(|x-y|)
  \frac{|\varphi(x)-\varphi(y)|^2}
  {\ep^{2-\alpha}|x-y|^\alpha}
  \,\xd x\,\xd y=
  \frac14\int_{\Omega\times\Omega}
  \tilde\rho_\ep(|x-y|)
  \frac{|\varphi(x)-\varphi(y)|^2}
  {|x-y|^2}
  \,\xd x\,\xd y\,,
  \]
  where
  \[
  \tilde\rho_\ep(r):=\rho_\ep(r)
  \frac{r^{2-\alpha}}{\ep^{2-\alpha}}\,, \quad
  r\geq0\,.
  \]
  The three statements are then a consequence of 
  the results in \cite{ponce04, ponce}
  and \cite{DST} provided
  that the new sequence $(\tilde\rho_\ep)_{\ep}$
  of mollifiers satisfies the following 
  properties:
  \begin{align*}
      &\int_0^{+\infty}
      \tilde\rho_\ep(r)r^{d-1}\,\xd r
      =\frac2{C_d} \quad\forall\,\ep>0\,,\\
      &\lim_{\ep\to0}\int_{\delta}^{+\infty}
      \tilde\rho_\ep(r)r^{d-1}\,\xd r=0\quad
      \forall\,\delta>0\,.
  \end{align*}
  The first condition follows from {\bf H2} and
the computation
\[
  \int_0^{+\infty}
  \tilde\rho_\ep(r)r^{d-1}\,\xd r =
  \frac1{\ep^{d+2-\alpha}}\int_0^{+\infty}
  \rho(r/\ep)r^{d+1-\alpha}\,\xd r=
  \int_0^{+\infty}\rho(s)s^{d+1-\alpha}\,\xd s=\frac2{C_d}\,,
\]
while the second condition follows from
the equality
\[
  \int_\delta^{+\infty}
  \tilde\rho_\ep(r)r^{d-1}\,\xd r =
  \int_{\delta/\ep}^{+\infty}\rho(s)s^{d+1-\alpha}\,\xd s
\]
and the dominated convergence theorem.
\end{proof}

Lemma~\ref{lem:conv_E} above 
can be used to study the 
asymptotic behaviour of $DE_\ep$
as $\ep\to0$. To this aim, we introduce 
the operators
\[
  B_\ep:H\to H\,, \qquad
  B_\ep(\varphi):=DE_\ep(\varphi)
  =(J_\ep*1)\varphi - J_\ep*\varphi\,,
  \quad\varphi\in H\,,
\]
and
\[
  B:V\to V^*\,,\qquad
  \langle B(\varphi),\zeta\rangle:=
  \int_\Omega\nabla\varphi(x)\cdot
  \nabla\zeta(x)\,\xd x\,,
  \quad\varphi,\zeta\in V\,.
\]
Lemma~\ref{lem:conv_E} 
implies that 
$B_\ep(\varphi)\rightharpoonup 
B(\varphi)$ in $V^*$
for every 
$\varphi\in V$ as $\ep\searrow0$.
As an immediate consequence,
for every sequence
$(\varphi_\ep)_\ep\subset V$
such that $\varphi_\ep\to\varphi$
in $V$, there 
holds $B_\ep(\varphi_\ep)
\rightharpoonup B(\varphi)$ in $V^*$.
While this identification result 
is surely enough for standard 
purposes, it is not entirely 
satisfactory. Indeed, 
in several applications the 
strong convergence
$\varphi_\ep\to\varphi$ in $V$ is 
far beyond reach. The main reason 
is that the regularity 
$\varphi_\ep\in V$ can be 
generally obtained 
at $\ep>0$ fixed, but not uniformly 
in $\ep$ (as the gradient 
of $J_\ep$ blows up as $\ep\to0$,
as shown in the proof of
Lemma~\ref{lem:rho}).\\

The following proposition provides a
far more general sufficient condition
to identify the limit of the 
operators $B_\ep$, without requiring any control in the space $V$.
The main idea is to observe that for any suitable
test function $\zeta$, by symmetry of $B_\ep$, we can write
$\langle B_\ep(\varphi_\ep),
\zeta\rangle=
\langle B_\ep(\zeta),
\varphi_\ep\rangle$. Hence, if only
$\varphi_\ep\to\varphi$ in $H$ and
$\zeta\in W$ is such that 
$B_\ep(\zeta)\rightharpoonup
-\Delta\zeta$ in $H$, 
then we can still conclude.
However, it is absolutely not trivial 
to show that $(B_\ep(\zeta))_\ep$ is
bounded in $H$ for a specific class of 
functions $\zeta$: in particular, 
the natural choice $\zeta\in W$
does not seem to work.
The main novelty of the following proposition 
consists in showing that, nevertheless,
any $\zeta\in W$ can be suitably
approximated by a sequence 
$(\zeta_\ep)_\ep$ for which 
$\zeta_\ep\to\zeta$ in $H$ and
$B_\ep(\zeta_\ep)
\to-\Delta\zeta$ in $H$,
and that this is enough to conclude.

\begin{proposition}
  \label{lem:ident_B}
  Let $(\varphi_\ep)_\ep\subset H$ and
  $\varphi\in H$ be such that 
  $\varphi_\ep\to\varphi$ in $H$
  and $(E_\ep(\varphi_\ep))_\ep$
  is uniformly bounded. Then,
  $\varphi\in V$ and 
  $B_\ep(\varphi_\ep)\rightharpoonup
  B(\varphi)$ in $V^*$.
\end{proposition}
\begin{proof}
  Since $(E_\ep(\varphi_\ep))_\ep$
  is uniformly bounded in $\ep$,
  the family $(B_\ep(\varphi_\ep))_\ep$
  is uniformly bounded in $V^*$.
  Indeed, for every $\psi\in V$ there holds
  $$|\langle B_\ep(\varphi_\ep),\psi\rangle_{V}|\leq 2\sqrt{E_\ep(\varphi_\ep)}\sqrt{E_\ep(\psi)}\leq C\|\psi\|_V,$$
  where the last step is a direct consequence of \cite[Theorem 1]{BBM}.
  
  Hence, there exists $\eta\in V^*$ such that $B_\ep(\varphi_\ep)
  \rightharpoonup\eta$ weakly* in $V^*$
  along a subsequence. 
  Moreover, the uniform
  boundedness of
  $(E_\ep(\varphi_\ep))_\ep$ also
  ensures, thanks to
  Lemma~\ref{lem:conv_E}, that 
  $\varphi\in V$.
  Now, we have to show that
  $\eta=B(\varphi)$. To this end, 
  let $\zeta\in W$ be arbitrary
  and define
  \[
  \zeta_\ep:=\zeta_\Omega +
  \tilde\zeta_\ep\,, \qquad
  \ep\tilde\zeta_\ep +
  B_\ep(\tilde\zeta_\ep) = 
  -\Delta \zeta\,.
  \]
  Testing by $\tilde\zeta_\ep$
  we get, by the Young inequality, 
  \[
  \ep\|\tilde\zeta_\ep\|_H^2
  +2E_\ep(\tilde\zeta_\ep)=
  (-\Delta\zeta,
  \tilde\zeta_\ep)_{H}\leq
  \delta\|\tilde\zeta_\ep\|_H^2
  +\frac1{4\delta}\|\Delta\zeta\|_H^2
  \]
  for every $\delta>0$. By definition
  $(\tilde\zeta_\ep)_\Omega=0$, owing to the symmetry of the kernel $J_\ep$ and the definition of $W$.
  Thus, from the generalized
  Poincar\'e inequality contained in 
  \cite[Theorem 1.1]{ponce04}, there exists 
  $c_p>0$, independent of $\ep$, 
  such that 
  $\|\tilde\zeta_\ep\|_H^2\leq c_p
  E_\ep(\tilde\zeta_\ep)$. Choosing then $\delta$ sufficiently small
  (for example $\delta<2/c_p$),
  we deduce that there exists $M>0$
  independent of $\delta$ such that 
  \[
  \|\tilde\zeta_\ep\|_H^2
   + E_\ep(\tilde\zeta_\ep) \leq M\,.
  \]
  Thanks to the compactness result in
  \cite[Theorem~1.2]{ponce04},
  we deduce that there 
  exists $\tilde\zeta\in V$ such that 
  $\tilde\zeta_\ep\to \tilde \zeta$
  strongly in $H$. By comparison in the 
  equation for $\tilde\zeta_\ep$
  it follows that 
  $$B_\ep(\tilde\zeta_\ep)\to-\Delta \zeta\quad\text{strongly in }H.$$ Now, as $B_\ep=DE_\ep$ on $H$,  
  we have that 
  \[
  E_\ep(\tilde\zeta_\ep)
  +(B_\ep(\tilde\zeta_\ep),
  \zeta-\tilde\zeta_\ep)_{H}\leq
  E_\ep(\zeta).
  \]
  Letting $\ep\to0$, 
  this yields 
  thanks to Lemma~\ref{lem:conv_E}
  \[
  E(\tilde\zeta)
  +( -\Delta\zeta,
  z-\tilde\zeta)_{H}\leq
  E(\zeta) ,
  \]
  which implies in turn that 
  $\nabla\zeta=\nabla\tilde\zeta$ almost everywhere in $\Omega$.
  Consequently, recalling that 
  $\tilde\zeta_\Omega=0$, 
  we have that $\tilde\zeta=\zeta-
  \zeta_\Omega$.
  This implies in particular that 
  $\zeta_\ep\to\zeta$ in $H$, 
  and $B_\ep(\zeta_\ep)\to-\Delta\zeta$
  in $H$. Moreover, the symmetry of
  $B_\ep$ and the definition of $E_\ep$
  ensure also that 
  \begin{align*}
  E_\ep(\zeta_\ep-\zeta)&=
  \frac12( B_\ep(\zeta_\ep-\zeta), \zeta_\ep-\zeta)_H=
  E_\ep(\zeta_\ep)
  -\frac12( B_\ep(\zeta_\ep), \zeta)_H
  -\frac12( B_\ep(\zeta), \zeta_\ep)_H
  +E_\ep(\zeta)\\
  &=E_\ep(\zeta_\ep)
  -( B_\ep(\zeta_\ep), \zeta)_H
  +E_\ep(\zeta)\to
  \frac12\|\nabla\zeta\|_H^2
  -( -\Delta\zeta, \zeta)_H
  +\frac12\|\nabla\zeta\|_H^2=0\,.
  \end{align*}
  Now we can conclude. Indeed, 
  on the one hand we have 
  \[
  ( B_\ep(\varphi_\ep),
  \zeta_\ep)_H=
  ( B_\ep(\varphi_\ep),
  \zeta_\ep-\zeta)_H
  +( B_\ep(\varphi_\ep),
  \zeta)_H\to
  ( \eta,
  \zeta)_H
  \]
  since, by Lemma \ref{lem:conv_E} and the fact that
  $E_\ep(\varphi_\ep)\leq M$,
  \[
  |( B_\ep(\varphi_\ep),
  \zeta_\ep-\zeta)_H|\leq
  2\sqrt{E_\ep(\varphi_\ep)}
  \sqrt{E_\ep(\zeta_\ep-\zeta)}\leq
  2M\sqrt{E_\ep(\zeta_\ep-\zeta)}
  \to0\,.
  \]
  On the other hand, 
  \[
  ( B_\ep(\varphi_\ep),
  \zeta_\ep)_H=
  (\varphi_\ep,
  B_\ep(\zeta_\ep))_H
  \to( \varphi,
  -\Delta\zeta)_H=\langle
  B(\varphi),
  \zeta\rangle_V\,.
  \]
  Since $\zeta\in W$ is arbitrary 
  and $W$ is dense in $V$,
  we deduce that $\eta=B(\varphi)$,
  and that the convergence holds
  along the entire sequence by
  uniqueness
  of the limit, as required.
\end{proof}

We conclude this section with a technical result that will play a key role in the study of nonlocal-to-local asymptotics.
\begin{lemma}
  \label{lem:comp}
  For every $\delta>0$, there exist
  $C_\delta >0$ and $\ep_\delta>0$ such that,
  for every sequence 
  $(\varphi_\ep)_{\ep>0}\subset H$ it holds
  \[
  \|\varphi_{\ep_1}-\varphi_{\ep_2}\|_H^2
  \leq\delta\left(
  E_{\ep_1}(\varphi_{\ep_1})+
  E_{\ep_2}(\varphi_{\ep_2})\right)
  +C_\delta\|\varphi_{\ep_1}
  -\varphi_{\ep_2}\|_{V^*}^2
  \quad\forall\,\ep_1,\ep_2\in(0,\ep_\delta)\,.
  \]
\end{lemma}
\begin{proof}
  The proof follows from the compactness 
  property in Lemma~\ref{lem:conv_E}, as in
  \cite[Lem.~3]{DST} and \cite[Lem.~4]{DRST}.
\end{proof}

\section{The nonlocal problem}
\label{sec:nonlocal}
This section is devoted to 
the proof of Theorem~\ref{thm1}. Our analysis is performed for a fixed $\ep>0$. In what follows, we will denote by $M_\ep$ a generic constant independent of $\lambda$. 

We consider the auxiliary problem
\begin{align}
    \label{eq1_app}
    \partial_t u_{\ep\lambda}-\Delta\mu_{\ep\lambda}
    =-\div(u_{\ep\lambda}{\bf v}) 
    \qquad&\text{in } (0,T)\times\Omega\,,\\
    \label{eq2_app}
    \mu_{\ep\lambda}=
    (J_\ep*1)u_{\ep\lambda} - J_\ep*u_{\ep\lambda} +
    \gamma_\lambda(u_{\ep\lambda})
    +\Pi(u_{\ep\lambda})
    \qquad&\text{in } (0,T)\times\Omega\,,\\
    \label{eq3_app}
    \partial_{\bf n} \mu_{\ep\lambda}=0
    \qquad&\text{in } (0,T)\times\partial\Omega\,,\\
    \label{eq4_app}
    u_{\ep\lambda}(0)=u_{0,\ep} \qquad&\text{in } \Omega\,,
\end{align}
where, for every $\lambda>0$, $\gamma_\lambda:\mathbb R\to \mathbb R$ is the
Yosida approximation of $\gamma$, having Lipschitz constant $1/\lambda$. 

Since $J_\ep\in W^{1,1}(\mathbb R^d)$, from classical results 
(see \cite{della-grass, gal-gior-grass-NLCH}) it follows 
that the approximated problem \eqref{eq1_app}--\eqref{eq4_app}
admits a unique solution 
$(u_{\ep\lambda}, \mu_{\ep\lambda})$,
with
\[
  u_{\ep\lambda} \in H^1(0,T; V^*)\cap 
  L^2(0,T; V)\,,\qquad
  \mu_{\ep\lambda} \in L^2(0,T; V)\,.
\]

\subsection{Uniform estimates}
In this subsection, we show that the solutions to \eqref{eq1_app}--\eqref{eq4_app} satisfy uniform estimates independent of $\lambda$. Recall that here
$\ep>0$ is fixed.

Testing \eqref{eq1_app} by $\mu_{\ep\lambda}$, 
\eqref{eq2_app} by $\partial_t u_{\ep\lambda}$,
and taking the difference, by Lemma~\ref{lem:diff} we obtain
\begin{align*}
    &\int_{Q_t}|\nabla\mu_{\ep\lambda}(s,x)|^2\,\xd x\,\xd s+
    E_\ep(u_{\ep\lambda}(t))+
    \int_\Omega(\hat\gamma_\lambda(u_{\ep\lambda}(t,x))
    +\hat\Pi(u_{\ep\lambda}(t,x)))\,\xd x\\
    &=E_\ep(u_{0,\ep})+
    \int_\Omega(\hat\gamma_\lambda(u_{0,\ep}(x))
    +\hat\Pi(u_{0,\ep}(x)))\,\xd x +
    \int_{Q_t}u_{\ep\lambda}(s,x){\bf v}(s,x)\cdot
    \nabla \mu_{\ep\lambda}(s,x)\,\xd x\,\xd s\,,
\end{align*}
for every $t\in (0,T)$. From the Young inequality, we infer that 
\begin{align}\label{eq:est-extra1}
\begin{aligned}
    &\frac12\int_{Q_t}|\nabla\mu_{\ep\lambda}(s,x)|^2\,\xd x\,\xd s+
    E_\ep(u_{\ep\lambda}(t))+
    \int_\Omega(\hat\gamma_\lambda(u_{\ep\lambda}(t,x))
    +\hat\Pi(u_{\ep\lambda}(t,x)))\,\xd x\\
    &\leq E_\ep(u_{0,\ep})+ \|\hat\gamma(u_{0,\ep})
    +\hat\Pi(u_{0,\ep})\|_{L^1(\Omega)}
    +\frac12\int_{0}^t
    \|{\bf v}(s)\|_{L^\infty(\Omega)}^2
    \|u_{\ep\lambda}(s)\|^2_{H}\,\xd s\,.
    \end{aligned}
\end{align}
By the generalized Poincar\'e inequality in \cite{ponce04}, we deduce the existence of a constant 
$c_p>0$, independent of $\lambda$, such that 
\[
  \|u_{\ep\lambda} - (u_{\ep\lambda})_\Omega\|_H^2
  \leq c_p E_\ep(u_{\ep\lambda})\,.
\]
Since $(u_{\ep\lambda})_\Omega=(u_{0,\ep})_\Omega$ from
\eqref{eq1_app}, summing and subtracting 
$(u_{0,\ep})_\Omega$ in the third term on
the right-hand side of \eqref{eq:est-extra1}, we have
\begin{align*}
    \frac12\int_{Q_t}|\nabla\mu_{\ep\lambda}(s,x)|^2\,\xd x\,\xd s&+
    E_\ep(u_{\ep\lambda}(t))+
    \int_\Omega(\hat\gamma_\lambda(u_{\ep\lambda}(t,x))
    +\hat\Pi(u_{\ep\lambda}(t,x)))\,\xd x\\
    &\leq E_\ep(u_{0,\ep})+ \|\hat\gamma(u_{0,\ep})
    +\hat\Pi(u_{0,\ep})\|_{L^1(\Omega)}
    +\|{\bf v}\|^2_{L^2(0,T; L^\infty(\Omega))}
    \|u_{0,\ep}\|_H^2\\
    &+{c_p}\int_{0}^t
    \|{\bf v}(s)\|_{L^\infty(\Omega)}^2
    E_\ep(u_{\ep\lambda}(s))\,\xd s\,.
\end{align*}
By the Gronwall lemma and the assumptions on 
$u_{0,\ep}$ and ${\bf v}$, 
there exists a constant $M_\ep>0$, independent of $\lambda$,
such that 
\begin{equation}
  \label{est1}
  \|\nabla\mu_{\ep\lambda}\|_{L^2(0,T; H)}^2
  +\|E_\ep(u_{\ep\lambda})\|_{L^\infty(0,T)}+
  \|u_{\ep\lambda}\|_{L^\infty(0,T; H)}^2\leq M_\ep\,.
\end{equation}
By comparison in \eqref{eq1_app} we also infer that 
\begin{equation}
    \label{est2}
    \|\partial_t u_{\ep\lambda}\|_{L^2(0,T; V^*)}\leq M_\ep\,.
\end{equation}

Now, testing \eqref{eq1_app} by 
$\mathcal N(u_{\ep\lambda}-(u_{0,\ep})_\Omega)$, equation
\eqref{eq2_app} 
by $u_{\ep\lambda}-(u_{0,\ep})_\Omega$, taking the difference, and using the symmetry of the kernel $J_\ep$, yields
\begin{align}
\label{eq:est-extra2}
\begin{aligned}
  &\langle\partial_t u_{\ep\lambda},
  \mathcal N(u_{\ep\lambda}-(u_{0,\ep})_\Omega)
  \rangle_V+ 2 E_\ep(u_{\ep\lambda})
  +\int_\Omega\gamma_\lambda(u_{\ep\lambda}(x))
  (u_{\ep\lambda}(x)-(u_{0,\ep})_\Omega)\,\xd x\\
  &=\int_\Omega u_{\ep\lambda}(x)
  {\bf v}(x)\cdot\nabla\mathcal N
  (u_{\ep\lambda}(x)-(u_{0,\ep})_\Omega)\,\xd x
  -\int_\Omega
  \Pi(u_{\ep\lambda}(x))
  (u_{\ep\lambda}(x)-(u_{0,\ep})_\Omega)\,\xd x\,.
  \end{aligned}
\end{align}
The first two terms on the left-hand side of \eqref{eq:est-extra2}
are bounded in $L^2(0,T)$ independently 
of $\lambda$ thanks to the estimates \eqref{est1}--\eqref{est2}. Owing to
the
properties of $\mathcal N$ and the H\"older 
inequality, the right-hand side of \eqref{eq:est-extra2}
can be estimated from above  by
\[
  \|u_{\ep\lambda}\|_H
  \|{\bf v}\|_{L^\infty(\Omega)}
  \|u_{\ep\lambda}-(u_{0,\ep})_\Omega\|_{V^*}
  +\|\Pi(u_{\ep\lambda})\|_H
  \|u_{\ep\lambda}-(u_{0,\ep})_\Omega\|_{H}\,,
\]
hence they are bounded in $L^2(0,T)$
thanks to the Lipschitz-continuity of $\Pi$, assumption \textbf{H4},
and again \eqref{est1}. 
Moreover, since $(u_{0,\ep})_\Omega
\in\operatorname{Int}D(\gamma)$, there exist
two constants $c_\ep, c_\ep'>0$, 
independent of $\lambda$,
such that 
\[
  \int_\Omega\gamma_\lambda(u_{\ep\lambda}(x))
  (u_{\ep\lambda}(x)-(u_{0,\ep})_\Omega)\,\xd x
  \geq c_\ep
  \|\gamma_\lambda(u_{\ep\lambda})\|_{L^1(\Omega)}
  -c_\ep'\,.
\]
By comparison,
we deduce that
\[
  \|\gamma_\lambda(u_{\ep\lambda})\|_{L^2(0,T; L^1(\Omega))}\leq M_\ep\,.
\]
This yields, after integrating equation
\eqref{eq2_app} on $\Omega$ and using \eqref{est1}, the estimate
\[
  \|(\mu_{\ep\lambda})_\Omega\|_{L^2(0,T)}\leq M_\ep\,,
\]
which together with \eqref{est1} implies
\begin{equation}
    \label{est3}
    \|\mu_{\ep\lambda}\|_{L^2(0,T; V)}\leq M_\ep\,.
\end{equation}

Since $J_\ep\in W^{1,1}(\mathbb R^d)$,
we directly obtain
\begin{align}
\label{eq:est-extra3}
\begin{aligned}
  &\|J_\ep*1\|_{L^\infty(\Omega)}\leq
  \|J_\ep\|_{L^1(\mathbb R^d)}\,,\\
  &\|\nabla(J_\ep*1)\|_{L^\infty(\Omega)}=
  \|(\nabla J_\ep)*1\|_{L^\infty(\Omega)}\leq
  \|\nabla J_\ep\|_{L^1(\mathbb R^d)}\,,\\
  &\|J_\ep*u_{\ep\lambda}\|_H\leq
  \|J_\ep\|_{L^1(\mathbb R^d)}\|u_{\ep\lambda}\|_H
  \,,\\
  &\|\nabla(J_\ep*u_{\ep\lambda})\|_H=
  \|(\nabla J_\ep)*u_{\ep\lambda}\|_H\leq
  \|\nabla J_\ep\|_{L^1(\mathbb R^d)}
  \|u_{\ep\lambda}\|_H\,,
  \end{aligned}
\end{align}
which imply, together with \eqref{est1}, that 
\begin{equation}
    \label{est4}
    \|J_\ep*u_{\ep\lambda}\|_{L^\infty(0,T; V)}
    \leq M_\ep\,.
\end{equation}
Differentiating \eqref{eq2_app} and
testing it against $\nabla u_{\ep\lambda}$, we have
\begin{align*}
  &\int_\Omega
  \left(\gamma_\lambda'(u_{\ep\lambda}(x))
  +\Pi'(u_{\ep\lambda}(x))
  +(J_\ep*1)(x)\right)
  |\nabla u_{\ep\lambda}(x)|^2\,\xd x\\
  &=\int_\Omega 
  \left(\nabla\mu_{\ep\lambda}(x)+
  \nabla(J_\ep*u_{\ep\lambda})(x)
  -\nabla(J_\ep*1)(x)\right)
  \cdot\nabla u_{\ep\lambda}(x)
  \,\xd x\,.
\end{align*}
Recalling now assumption {\bf H3} and using the Young 
inequality yields
\begin{equation}
    \label{eq:est-extra4}
  c_\ep^0\|\nabla u_{\ep\lambda}\|_H^2\leq
  \frac{c_\ep^0}2\|\nabla u_{\ep\lambda}\|_H^2
  +\frac{3c_\ep^0}{2}\|\nabla\mu_{\ep\lambda}\|_H^2
  +\frac{3c_\ep^0}{2}\|\nabla(J_\ep*u_{\ep\lambda})\|_H^2
  +\frac{3c_\ep^0}{2}
  \|\nabla J_\ep\|_{L^1(\mathbb R^d)}^2\,.
\end{equation}
By combining \eqref{eq:est-extra4} with \eqref{est1} and \eqref{est4}, we infer that
\begin{equation}
    \label{est5}
    \|u_{\ep\lambda}\|_{L^2(0,T; V)}\leq M_\ep\,.
\end{equation}
Since $\nabla((J_\ep*1)u_{\ep\lambda})=
(\nabla (J_\ep*1))u_{\ep\lambda}+
(J_\ep*1)\nabla u_{\ep\lambda}$,  estimates \eqref{eq:est-extra4} and
\eqref{est5} 
yield
\begin{equation}
    \label{est6}
    \|(J_\ep*1)u_{\ep\lambda}\|_{L^2(0,T; V)}
    \leq M_\ep\,.
\end{equation}
By comparison in equation \eqref{eq2_app}, in view of \eqref{est3}--\eqref{est6} and the Lipschitz regularity of $\Pi$,
we deduce the bound
\begin{equation}
    \label{est7}
    \|\gamma_\lambda(u_{\ep\lambda})\|_{L^2(0,T; V)}
    \leq M_\ep\,.
\end{equation}

\subsection{Passage to the limit as $\lambda\searrow0$}
We conclude by passing to the limit as $\lambda\searrow0$
in the approximating problem
\eqref{eq1_app}--\eqref{eq4_app}, with $\ep>0$ fixed.

Estimates \eqref{est1}--\eqref{est7} 
and the classical Aubin-Lions compactness results
(see \cite[Cor.~4]{simon}) ensure
that there exist
\[
  u_\ep\in H^1(0,T; V^*)\cap L^2(0,T; V)\,, \qquad
  \mu_\ep\in L^2(0,T; V)\,, \qquad
  \xi_\ep\in L^2(0,T; V)
\]
such that, as $\lambda\searrow0$,
\begin{align*}
    u_{\ep\lambda}\to u_\ep 
    \quad&\text{in } L^2(0,T; H)\,,\\
    u_{\ep\lambda}\rightharpoonup u_\ep 
    \quad&\text{in } H^1(0,T; V^*)\cap L^2(0,T; V)\,,\\
    \mu_{\ep\lambda}\rightharpoonup\mu_\ep
    \quad&\text{in } L^2(0,T; V)\,,\\
    \gamma_\lambda(u_{\ep\lambda})\rightharpoonup\xi_\ep
    \quad&\text{in } L^2(0,T; V)\,.
\end{align*}
The strong-weak closure of the maximal monotone
operator $\gamma$ (see \cite[\S~2]{barbu-monot})
readily implies that $\xi_\ep\in\gamma(u_\ep)$
almost everywhere in $Q$, while the Lipschitz-continuity
of $\Pi$ ensures that
\[
  \Pi(u_{\ep\lambda})\to\Pi(u_\ep)
    \quad\text{in } L^2(0,T; H)\,.
\]
Moreover, since
$H^1(0,T; V^*)\cap L^2(0,T; V)\embed C^0([0,T]; H)$, 
we also have $u_\ep(0)=u_{0,\ep}$.
Finally, the properties of the convolution
and the strong convergence 
of $(u_{\ep\lambda})_\lambda$ yield
\begin{align*}
    (J_\ep*1)u_{\ep\lambda}\to 
    (J_\ep*1)u_\ep \quad&\text{in } L^2(0,T; H)\,,\\
    J_\ep*u_{\ep\lambda}\to 
    J_\ep*u_\ep \quad&\text{in } L^2(0,T; H)\,.
\end{align*}
Passing then to the weak limit in 
the approximated problem \eqref{eq1_app}--\eqref{eq4_app},
we complete the proof of Theorem~\ref{thm1}. The uniqueness of solution is obtained arguing analogously to~\cite{DRST,DST}.

\section{The nonlocal-to-local asymptotics}
\label{sec:nonlocal-to-local}
In this section, we perform the limit as $\ep\searrow 0$ in the nonlocal problem in Theorem~\ref{thm1},
and we prove Theorem~\ref{thm2}.

We first note that the assumptions of 
Theorem~\ref{thm2} ensure that the constants 
$(M_\ep)_{\ep>0}$ in estimates
\eqref{est1}, \eqref{est2}, and \eqref{est3}
are uniformly bounded in $\ep$.
Hence, by weak lower semicontinuity we infer that 
there exists a constant $M>0$, independent of $\ep$,
such that 
\begin{equation}\label{est1_ep}
  \|\mu_{\ep}\|_{L^2(0,T; V)}^2
  +\|E_\ep(u_{\ep})\|_{L^\infty(0,T)}+
  \|u_{\ep}\|_{H^1(0,T; V^*)\cap 
  L^\infty(0,T; H)}^2\leq M\,.
\end{equation}
Testing \eqref{eq2_NL} by $\xi_\ep$, 
rearranging the terms, and taking 
\eqref{est1_ep} into account yields
\begin{align}
\label{eq:extra-fin}
\begin{aligned}
  &\int_Q|\xi_\ep(s,x)|^2\,\xd x\,\xd s
  + \frac12\int_Q J_\ep(x-y)(\xi_\ep(s,x)-\xi_\ep(s,y))
  (u_\ep(s,x)-u_\ep(s,y))\,\xd x\,\xd y\,\xd s\\
  &=\int_Q(\mu_\ep-\Pi(u_\ep))(s,x)
  \xi_\ep(s,x)\,\xd x\,\xd s
  \leq\frac12\int_Q|\xi_\ep(s,x)|^2\,\xd x\,\xd s +M\,.
  \end{aligned}
\end{align}
The second term on the left hand side of \eqref{eq:extra-fin}
is nonnegative as $\xi_\ep\in\gamma(u_\ep)$, and because of \textbf{H3}.
Thus, we infer that 
\begin{equation}
    \label{est2_ep}
    \|\xi_\ep\|_{L^2(0,T; H)}\leq M\,.
\end{equation}
By comparison in \eqref{eq2_NL}, it follows that 
\begin{equation}
    \label{est3_ep}
    \|(J_\ep*1)u_\ep 
    - J_\ep*u_\ep\|_{L^2(0,T; H)}\leq M\,.
\end{equation}

Thanks to Aubin-Lions compactness 
results, estimates \eqref{est1_ep}--\eqref{est3_ep}
imply that there exist
\[
  u\in H^1(0,T; V^*)\cap L^\infty(0,T; H)\,,\qquad
  \mu\in L^2(0,T; V)\,, \qquad
  \xi,\eta\in L^2(0,T; H)\,,
\]
such that, as $\ep\searrow0$,
\begin{align*}
    u_\ep\to u \quad&\text{in } C^0([0,T]; V^*)\,,\\
    u_\ep\wstarto u 
    \quad&\text{in } 
    L^\infty(0,T; H)\cap 
    H^1(0,T;V^*)\,,\\
    \mu_\ep\rightharpoonup \mu 
    \quad&\text{in } L^2(0,T; V)\,,\\
    \xi_\ep\rightharpoonup \xi 
    \quad&\text{in } L^2(0,T; H)\,,\\
    (J_\ep*1)u_\ep-J_\ep*1\rightharpoonup \eta
    \quad&\text{in } L^2(0,T; H)\,.
\end{align*}
The strong convergence of $(u_\ep)_\ep$
 yields $u(0)=u_0$.
Furthermore, Lemma~\ref{lem:comp} 
and estimate \eqref{est1_ep} imply that 
for every $\delta>0$ there exists $C_\delta>0$
and $\ep_\delta>0$
such that 
\begin{align*}
    \|u_{\ep_1}-u_{\ep_2}\|_{C^0([0,T]; H)}^2&\leq
    \delta\|E_{\ep_1}(u_{\ep_1})+
    E_{\ep_2}(u_{\ep_2})\|_{L^\infty(0,T)}
    +C_\delta\|u_{\ep_1}
    -u_{\ep_2}\|_{C^0([0,T]; V^*)}^2\\
    &\leq 2M\delta + C_\delta \|u_{\ep_1}
    -u_{\ep_2}\|_{C^0([0,T]; V^*)}^2\,.
\end{align*}
Since $\delta>0$ is arbitrary and
$(u_\ep)_\ep$ converges strongly in 
$C^0([0,T]; V^*)$, we deduce that
\[
  u_\ep\to u \quad\text{in } C^0([0,T]; H)\,.
\]
This readily implies that $\xi\in\gamma(u)$
by the strong-weak closure of $\gamma$, and that 
\[
  \Pi(u_\ep)\to \Pi(u) \quad\text{in } C^0([0,T]; H)
\]
by the Lipschitz-continuity of $\Pi$.

Passing to the limit in
the weak formulation
of \eqref{eq1_NL}--\eqref{eq4_NL} we infer,
by the dominated convergence theorem, that 
\[
  \langle\partial_t u,\varphi\rangle_V + 
  \int_\Omega\nabla\mu(x)\cdot\nabla\varphi(x)\,\xd x = 
  \int_\Omega u(x){\bf v}(x)\cdot\nabla\varphi(x)\,\xd x
\]
for every $\varphi\in V$, almost everywhere in $(0,T)$,
and that $\mu=\eta+\xi+\Pi(u)$.

It only remains to show that $u\in L^\infty(0,T; V)
\cap L^2(0,T; W)$ and $\eta=-\Delta u$.
To this end, note that Lemma~\ref{lem:conv_E}
and estimate \eqref{est1_ep}
imply that
\[
  \|E(u)\|_{L^\infty(0,T)}
  \leq\liminf_{\ep\to0}\|E_\ep(u_\ep)\|_{L^\infty(0,T)}
  \leq M\,,
\]
which ensures that $u\in L^\infty(0,T; V)$.

Moreover, by Lemma~\ref{lem:diff} we know that,
for every $\varphi\in L^2(0,T; V)$,
\[
  \int_0^TE_\ep(u_\ep(s))\,\xd s
  +\int_Q
  \left((J_\ep*1)u_\ep-J_\ep*u_\ep\right)
  (s,x)(\varphi-u_\ep)(s,x)\,\xd x\,\xd s
  \leq \int_0^TE_\ep(\varphi(s))\,\xd s\,.
\]
Since $u_\ep\to u$ in $C^0([0,T]; H)$, 
Lemma~\ref{lem:conv_E}, \cite[Theorem 1]{BBM}, and Fatou's lemma
imply that 
\[
  \frac12\int_Q|\nabla u(s,x)|^2\,\xd x\,\xd s
  +\int_Q\eta(s,x)
  (\varphi-u)(s,x)\,\xd x\,\xd s
  \leq \frac12\int_Q|\nabla \varphi(s,x)|^2\,\xd x\,\xd s
\]
for every $\varphi\in L^2(0,T; V)$. Hence, we deduce that
$\eta\in \partial E(u)$, so that 
\[
  \int_\Omega\eta(x)\varphi(x)\,\xd x=
  \int_\Omega\nabla u(x)\cdot\nabla\varphi(x)\,\xd x
  \quad\forall\,\varphi\in V\,.
\]
By the classical elliptic regularity theory we
infer that $u\in L^2(0,T; W)$, $\eta=-\Delta u$,
and $\partial_{\bf n}u=0$ almost everywhere on $\partial\Omega$, as
required. For more details, we refer the reader 
to \cite{DST}.
This completes the proof of Theorem~\ref{thm2}.

\section{Extensions and applications}
\label{sec:appl}

In this section we present an overview of some settings in which a direct adaptation of Theorems \ref{thm1} and \ref{thm2} provides existence and nonlocal-to-local convergence.\\

\textbf{Periodic boundary conditions}. This modeling assumption  is equivalent to consider $\Omega$ as the $d$-dimensional flat torus. Existence, regularity, and nonlocal-to-local convergence have been analyzed in~\cite{MRT18,DRST} for interaction kernels enjoying much weaker integrability assumptions, and under slightly different hypotheses on the convection velocity $\textbf{v}$. 
As a consequence of the periodicity of the problem, $J_\ep*1$ is constant over the domain and $\nabla (J_\ep*u_\ep)=J_\ep*\nabla u_\ep =\nabla J_\ep *u_\ep$ hold. The arguments in the proofs of Theorems \ref{thm1} and \ref{thm2} directly yield an extension of the results for regular kernel also to the setting of periodic boundary conditions.\\

\textbf{Viscous case and forcing term}. Consider the following viscous nonlocal problem
\begin{align}
    \label{eq1_NL2}
    \partial_t u_\ep-\Delta\mu_\ep
    =-\div (u_\ep{\bf v}) \qquad&\text{in } (0,T)\times\Omega\,,\\
    \label{eq2_NL2}
    \mu_\ep = \tau_\ep\partial_t u_\ep + (J_\ep*1)u_\ep - J_\ep*u_\ep + \gamma(u_\ep)
    +\Pi(u_\ep)-g_\ep \qquad&\text{in } (0,T)\times\Omega\,,\\
    \label{eq3_NL2}
    \partial_{\bf n} \mu_\ep=0
    \qquad&\text{in } (0,T)\times\partial\Omega\,,\\
    \label{eq4_NL2}
    u_\ep(0)=u_{0,\ep} \qquad&\text{in } \Omega\,,
\end{align}
and its local counterpart
\begin{align}
    \label{eq1_L2}
    \partial_t u-\Delta\mu
    =-\div (u{\bf v}) \qquad&\text{in } (0,T)\times\Omega\,,\\
    \label{eq2_L2}
    \mu = \tau \partial_t u-\Delta u + \gamma(u)
    +\Pi(u)-g \qquad&\text{in } (0,T)\times\Omega\,,\\
    \label{eq3_L2}
    \partial_{\bf n} u=\partial_{\bf n} \mu=0
    \qquad&\text{in } (0,T)\times\partial\Omega\,,\\
    \label{eq4_L2}
    u(0)=u_{0} \qquad&\text{in } \Omega\,.
\end{align}
For the nonlocal problem $\tau_\ep\geq 0$ is a viscosity coefficient and $g_\ep$ represents a distributed forcing term, while for the local one $\tau\geq 0$ is the limiting viscosity parameter. In particular, the choices $\tau>0$ and $\tau = 0$ correspond to the viscous case and pure case. The above setting, without the convection term in divergence form, has been studied in~\cite{DRST} under much weaker regularity assumptions on the interaction kernels. In particular, in~\cite{DRST} existence and convergence were proven by assuming $\tau_\ep>0$ for every $\ep>0$. 

Under $W^{1,1}$-regularity of the kernel, in problem~\eqref{eq1_NL2}--\eqref{eq4_NL2} (involving viscosity, a convective contribution, and a forcing term) we can prove existence also for $\tau_\ep=0$, and we can weaken  hypothesis \textbf{H3} on ${\bf v}$.

To be precise, the statement of the existence result reads as follows.
\begin{theorem}
  Let $\ep>0$ and 
  $\tau_\ep\geq 0$ be fixed. Then,
  for every $(u_{0,\ep},g_\ep)$ such that 
  $$u_{0,\ep}\in H, \quad \tau_\ep u_{0,\ep}\in V\,, \qquad\hat\gamma(u_{0,\ep})\in L^1(\Omega),\qquad
  (u_{0,\ep})_\Omega\in\operatorname{Int}
  D(\gamma),$$ and 
  $$g_\ep\in L^2(0,T;H) \quad\text{if } \tau_\ep>0\,, \qquad
  g_\ep\in H^1(0,T;H) \quad\text{if } \tau_\ep=0\,,
  $$
  there exists a triplet $(u_\ep, \mu_\ep, \xi_\ep)$, where $u_\ep$ is uniquely 
  determined,
  such that 
  \begin{align*}
      &u_\ep \in H^1(0,T; V^*)
      \cap L^2(0,T; V)\,,
      \qquad
      \mu_\ep \in L^2(0,T; V)\,,
      \qquad
      \xi_\ep\in L^2(0,T; V)\,,\\
      &\tau_\ep u_\ep \in 
      H^1(0,T;H)\cap L^\infty(0,T; V)\,,
      \qquad\tau_\ep\mu_\ep\in L^2(0,T; W)\,,\\
      &u_\ep(0)=u_{0,\ep}\\
      &\mu_\ep=\tau_\ep \partial_t u_\ep +(J_\ep*1)u_\ep - J_\ep*u_\ep + \xi_\ep+
      \Pi(u_\ep)-g_\ep\,,\\
      &\xi_\ep\in\gamma(u_\ep) \quad\text{a.e.~in } Q\,.
  \end{align*}
\end{theorem}
We observe that if $\tau_\ep$
is strictly positive we can relax the hypothesis on ${\bf v}$.
Since the computations are standard,
we omit here the details.
The nonlocal-to-local convergence 
reads exactly as in \cite[Thm.~3.3]{DST}.

\section{Connections with the theory of convergence of gradient flows}
\label{sec:SS}

We conclude this paper by highlighting the connection between the results in Theorem \ref{thm1} and the notion of evolutionary $\Gamma$-convergence for gradient flows introduced by E. Sandier and S. Serfaty in \cite{sand-serf} (see also \cite{serf} for an overview). In the absence of convection ($\mathbf{v}=0$), the nonlocal-to-local convergence in Theorem \ref{thm2}, as well as some of the a-priori estimates established in the proofs of Theorems \ref{thm1} and \ref{thm2} present many similarities with the abstract scheme developed in \cite[Theorem 1.4]{sand-serf}. For convenience of the reader, we recall this seminal result below, and proceed by showing the connection between the theorem below and the proof strategy in our setting.

\begin{theorem}[\cite{sand-serf}]
For every $\ep>0$, let $X_\ep$ be a Hilbert space, and let $$\mathcal{E}_\ep:X_\ep\to [0;+\infty)$$ be a $C^1$-functional. Assume also that $X$ is a Hilbert space, and that $$\mathcal{E}:X\to [0;+\infty)$$ is a $C^1$-functional. For every $\ep>0$, let $u_\ep^0\in X_\ep$, let $u^0\in X$, and assume that $u_\ep^0\to^\tau u^0$ in a suitable topology $\tau$, and that 
$$\lim_{\ep\to 0} \mathcal{E}_\ep(u_\ep^0)= \mathcal{E}(u^0).$$

Let $(u_\ep)_\ep$ be a family of conservative solutions for $\mathcal{E}_\ep$ on $[0,T)$ starting from $(u_\ep^0)_\ep$, namely such that $u_\ep\in H^1(0,T;X_\ep)$, for almost every $t\in [0,t)$, 
$$\partial_t u_\ep=-D_{X_\ep}\mathcal{E}_\ep(u_\ep)\in X_\ep,$$
and for every $t\in [0,T)$,
\begin{equation}
    \label{eq:en-cons}
\mathcal{E}_\ep(u_\ep^0)-\mathcal{E}_\ep(u_\ep(t))=\int_0^t\|\partial_t u_\ep(s)\|^2_{X_\ep}\,\xd s. 
\end{equation}
Assume also that the energy functionals $(\mathcal{E}_\ep)_\ep$ and $\mathcal{E}$ enjoy the following two properties:

\begin{align*}
&\textbf{C1}) \text{ if for every }\ep>0\text{ there holds } v_\ep\in H^1(0,T;X_\ep),\text{ and for a subsequence } v_\ep(t)\to^\tau v(t)\\
&\quad\text{ for every }t\in [0,T), \text{ then }v\in H^1(0,T;X),\text{ and for every }s\in [0,t),\\
&\qquad \qquad\qquad\qquad
\liminf_\ep \int_0^s \|\partial_t v_\ep(s)\|^2\xd s\geq \int_0^s \|\partial_t v(s)\|^2\xd s;\\
&\textbf{C2}) \text{ if for every }\ep>0\text{ there holds }w_\ep\in X_\ep,\text{ and }w_\ep\to^\tau w\text{ with }w\in X,\text{ then }\\
&\qquad \qquad\qquad\qquad\liminf_\ep \|D_{X_\ep}\mathcal{E}_\ep(w_\ep)\|_{X_\ep}\geq \|D_X \mathcal{E}(w)\|_{X}.
\end{align*}
Eventually, assume that there exists a map $u\in H^1(0,T;X)$ such that, for every $t\in[0,T)$, $u_\ep(t)\to^\tau u(t)$, and
$$\liminf_{\ep}\mathcal{E}_\ep(u_\ep(t))\geq \mathcal{E}(u(t)).$$
Then, $u$ is the solution in $[0,T)$ to the gradient flow associated to the energy functional $\mathcal{E}$, in the structure given by $X$, and with initial datum $u^0$.
\end{theorem}

The theorem above can not directly be applied in our framework. Indeed, despite the nonlocal energy functional $E_\ep$ in \eqref{eq:nl-en} is of class $C^1$ (see Lemma \ref{lem:diff}), the overall driving energy functional for the evolution equation in Theorem \ref{thm1} does not satisfy this regularity assumption, owing to the presence of the singular double-well potential. As pointed out in \cite[Section 2.2]{serf}, our convergence result should thus be read within the theory of evolutionary $\Gamma$-convergence of gradient flows in the more general setting of metric spaces. We refrain from introducing the full formalism here: we proceed by just briefly highlighting the main connections between the general theory in \cite{sand-serf, serf} and our proof strategy.

We first point out that the solutions to the nonlocal Cahn-Hilliard equations constructed in Theorem \ref{thm1}
(with ${\bf v}=0$) satisfy 
$u_\ep\in H^1(0,T;V^*)$ and 
$-\partial_t u_\ep\in 
\partial_{V^*}
\mathcal{E}_\ep^{NL}(u_\ep)$ as elements of $V^\ast$, where 
the subdifferential $\partial_{V^*}$
is intended in the dual space $V^*$,
$\mathcal{E}_\ep^{NL}$ is the functional defined in \eqref{eq:nl-en}, and the double-well potential $F$ is defined as $F=\hat{\gamma}+\hat{\Pi}$ (see also \textbf{H3}).

Additionally, the fact that $u_\ep$ is a conservative solution for $\mathcal{E}_{\ep}^{NL}$ in the sense of \eqref{eq:en-cons} with the choice 
$X_\ep=X=V^*$
follows directly by testing \eqref{eq:grad-flow} with 
$(-\Delta)^{-1}
\partial_t u_\ep$ and integrating the resulting equation with respect to time. Consider now a sequence of initial data $(u_{0,\ep})_\ep$ as in Theorem \ref{thm2}. In view of Theorem \ref{thm2} and Lemma \ref{lem:conv_E}, we deduce that there exists $u\in C([0,T];H)$ for which

$$\liminf_\ep \mathcal{E}_\ep^{NL}(u_\ep(t))\geq E(u(t))+\int_\Omega F(u(t,x))\xd x\quad\text{for every }t\in [0,T).$$
Eventually, conditions \textbf{C1}) and \textbf{C2}) follow directly by \eqref{eq1-f} and \eqref{eq2-f}. A crucial peculiarity of our setting is in the fact that, in order to guarantee the convergence of the gradient flows, it is also necessary to prove \eqref{eq3-f}. The thesis follows then in Theorem \ref{thm2} owing to the closure of the subdifferential of the double-well potential with respect to the above convergences.
\section{Conclusions}
\label{sec:con}
We have presented an existence result for a nonlocal Cahn-Hilliard equation with $W^{1,1}$ interaction kernel, under Neumann boundary conditions, and keeping track of the effects of a convective term in divergence form. Relying on the a-priori estimates identified in the proof of the existence result, we have shown convergence to a local Cahn-Hilliard equation as the kernel approaches a Dirac delta. This is the first nonlocal-to-local asymptotics for nonlocal Cahn-Hilliard equations satisfying homogeneous Neumann conditions and in the absence of regularizing viscous terms. Eventually, we have presented some extensions to the setting of periodic boundary conditions and to that of viscous Cahn-Hilliard equations, as well as a comparison with the classical theory of evolutionary convergence of gradient flows.

\section*{Acknowledgements}

E.D and L.T. have been supported by the Austrian Science Fund (FWF) project F 65. E.D. has been funded by the Austrian Science Fund (FWF) project V 662 N32. The research of E.D. has been additionally supported from the Austrian Science Fund (FWF) through the grant I 4052 N32, and from BMBWF through the OeAD-WTZ project CZ04/2019. 

\bibliographystyle{abbrv}
\bibliography{ref}
\end{document}